\documentclass{article}
\title{The Axiom of Multiple Choice and Models for Constructive Set Theory}
\author{Benno van den Berg\footnote{ILLC, Universiteit van Amsterdam, P.O. Box 94242, 1090 GE Amsterdam. E-mail: B.vandenBerg3@uva.nl. The first author was supported by the Netherlands Organisation for Scientific Research while working on the research reported here.} \, \& \, Ieke Moerdijk\footnote{Radboud Universiteit Nijmegen, Institute for Mathematics, Astrophysics, and
Particle Physics, Heyendaalseweg 135, 6525 AJ Nijmegen, the Netherlands.
E-mail: i.moerdijk@math.ru.nl.} }
\date{26 September, 2013}

\usepackage{ast2}

\def\epi{\ensuremath{\twoheadrightarrow}}

\begin{document}

\maketitle

\begin{abstract}
\noindent
We propose an extension of Aczel's constructive set theory {\bf CZF} by an axiom for inductive types and a choice principle, and show that this extension has the following properties: it is interpretable in Martin-L\"of's type theory (hence acceptable from a constructive and generalised-predicative standpoint). In addition, it is strong enough to prove the Set Compactness Theorem and the results in formal topology which make use of this theorem. Moreover, it is stable under the standard constructions from algebraic set theory, namely exact completion, realizability models, forcing as well as more general sheaf extensions. As a result, methods from our earlier work can be applied to show that this extension satisfies various derived rules, such as a derived compactness rule for Cantor space and a derived continuity rule for Baire space. Finally, we show that this extension is robust in the sense that it is also reflected by the model constructions from algebraic set theory just mentioned.
\end{abstract}

\section{Introduction}

There is a distinctive stance in the philosophy of mathematics which is usually called ``generalised predicativity''. It is characterised by the fact that it does not accept non-constructive and impredicative arguments, but it does allow for the existence of a wide variety of inductively defined sets. Martin-L\"of's type theory \cite{martinlof84} expresses this stance in its purest form. For the development of mathematics, however, this system has certain drawbacks: the type-theoretic formalism is involved and requires considerable time to get accustomed to, and the lack of extensionality leads to difficult conceptual problems. Aczel's interpretation of his constructive set theory {\bf CZF} in Martin-L\"of's type theory \cite{aczel78} overcomes both problems: the language of set theory is known to any mathematician and {\bf CZF} incorporates the axiom of extensionality. For this reason, {\bf CZF} has become the standard reference for a set-theoretic system expressing the ``generalised-predicative stance''.

It turns out, however, that {\bf CZF} is not quite strong enough to formalise all the mathematics which one would like to be able to formalise in it: there are results, in particular in formal topology, which can be proved in type theory and are perfectly acceptable from a generalised-predicative perspective, but which go beyond {\bf CZF}. There seem to be essentially two reasons for this: first of all, type theory incorporates the ``type-theoretic axiom of choice'' and secondly, Martin-L\"of type theory usually includes W-types which allow one to prove the existence of more inductively defined sets than can be justified in {\bf CZF} alone. To address this, we will suggest in this paper an extension of ${\bf CZF}$ which includes a form of choice and W-types, so that in it one can develop formal topology, while at the same time having good model-theoretic properties.

Let us take the second point first. Already in 1986, Peter Aczel suggested what he called the Regular Extension Axiom {\bf (REA)} to address this issue \cite{aczel86}. The main application of {\bf (REA)} is that it allows one to prove the ``Set Compactness Theorem'', which is important in formal topology (see \cite{aczel06, bergmoerdijk10c}), but not provable in {\bf CZF} proper. Here we suggest to take the axiom {\bf (WS)} instead: for every function $f: B \to A$ the associated W-type $W(f)$ is a set. (This is not the place to review the basics of W-types, something which we have already done on several occasions: see, for example, \cite{bergmoerdijk10b}.) One advantage of this axiom over {\bf (REA)} is that it directly mirrors the type theory. In addition, {\bf (WS)} is easy to formulate in the categorical framework of algebraic set theory, so that one may use this extensive machinery to establish its basic preservation properties (such as stability under exact completion, realizability and sheaves), whereas for {\bf (REA)} such a formulation does not seem to be possible. It has been claimed, quite plausibly, that {\bf (REA)} has similar stability properties, but we have never seen a proof of this claim.

As for the lack of choice in {\bf CZF}, the axiom which would most directly mirror the type theory would be the ``presentation axiom'', which says that the category of sets has enough projectives. The problem with this axiom, however, is that is not stable under taking sheaves. Precisely for this reason, Erik Palmgren together with the second author introduced in \cite{moerdijkpalmgren02} an axiom called the Axiom of Multiple Choice {\bf (AMC)}, which is implied by the existence of enough projectives and is stable under sheaves. This axiom (which we will discuss towards the end of this paper) is a bit involved and it turns out that on almost all occasions where one would like to use this axiom a slightly weaker and simpler principle suffices. This weaker principle is:
\begin{quote}
For any set $X$ there is a \emph{set} $\{ p_i: Y_i \twoheadrightarrow X \, : \, i \in I \}$
of surjections onto $X$ such that for any surjection $p: Y \twoheadrightarrow X$
onto $X$ there is an $i \in I$ and a function $f: Y_i \to Y$ such that $p \circ f =
p_i$.
\end{quote}
It is \emph{this} axiom which we will call {\bf (AMC)} in this paper, whereas we will refer to the original formulation in \cite{moerdijkpalmgren02} as ``strong {\bf (AMC)}''. (Independently from us, Thomas Streicher hit upon the same principle in \cite{streicher05}, where it was called TTCA$_f$; on the nLab, http://ncatlab.org, the principle is called WISC.)

To explain the name ``Axiom of Multiple Choice'', we remark that for a surjection $p: Y \to X$, a \emph{choice function} is a section $f: X \to Y$ of $p$. On the other hand, a \emph{multi-valued choice function} is a function $f: Y'\to Y$, defined on a cover $p': Y'\to X$ of $X$, for which $p f = p'$. Our axiom provides a (necessarily nonempty) family of domains sufficient to find such a multi-valued choice function for any surjection $p: Y \to X$.

So this is our proposal: extend the theory {\bf CZF} with the combination of {\bf (WS)} and {\bf (AMC)}. The resulting theory has the following properties:
\begin{enumerate}
\item It is validated by Aczel's interpretation in Martin-L\"of's type theory (with one universe closed under W-types) and therefore acceptable from a gene\-ra\-lised-predicative perspective.
\item The theory is strong enough to prove the Set Compactness Theorem and to develop that part of formal topology which relies on this result.
\item The theory is stable under the key constructions from algebraic set theory, such as exact completion, realizability and sheaves.
\end{enumerate}
It is the purpose of this paper to prove these facts. As a result, {\bf CZF} + {\bf (WS)} + {\bf (AMC)} will be the first (and so far only) theory for which the combination of these properties has been proved. And as a consequence of stability, the methods from \cite{bergmoerdijk10c} are applicable to it and one can show:
\begin{enumerate}
 \item[4.] The theory satisfies various derived rules, such as the derived Fan Rule and the derived Bar Induction Rule.
\end{enumerate}

Moreover, we will show that the theory has a certain robustness about it. Indeed, assuming that the ground model for {\bf CZF} satisfies {\bf (AMC)} and {\bf (WS)}, it is impossible to use the standard model-theoretic techniques to prove independence of {\bf (AMC)} and {\bf (WS)} from {\bf CZF}. To express this more formally, let us say that an axiom $\varphi$ in the language of ${\bf CZF}$ is \emph{reflected} by sheaf extensions (for example), if for any {\bf CZF}-model \ct{E}, the axiom $\varphi$ holds in \ct{E} as soon as it holds in some sheaf extension $\ct{E}'$ of \ct{E}. Then as a fourth property of our theory {\bf CZF} + {\bf (WS)} + {\bf (AMC)} we have
\begin{enumerate}
\item[5.] The theory is reflected by the model constructions of exact completion, realizability and sheaves.
\end{enumerate}

It should be noted that establishing the first property for {\bf CZF} + {\bf (WS)} + {\bf (AMC)} is quite easy, because stronger axioms are verified by the type-theoretic interpretation: {\bf (REA)} can be interpreted (that was the main result of \cite{aczel86}) and {\bf (REA)} implies {\bf (WS)} (see  \cite[page 5--4]{aczelrathjen01}), while {\bf (AMC)} is an obvious consequence of the presentation axiom which is validated by the type-theoretic interpretation (see \cite{aczel82}). Therefore it remains to establish the other properties in the list.

The contents of this paper are therefore as follows. First, we will show in Section 2 that the Set Compactness Theorem follows from the combination of {\bf (WS)} and {\bf (AMC)}. Then we will proceed to show that these axioms are stable under and reflected by exact completion (Section 3), realizability (Section 4) and sheaves (Section 5). Throughout these sections we assume familiarity with the framework for algebraic set theory developed in \cite{bergmoerdijk07a, bergmoerdijk08, bergmoerdijk09, bergmoerdijk10b}. Finally, in Section 6 we will discuss the relation of our present version of {\bf (AMC)} with the earlier and stronger formulation from \cite{moerdijkpalmgren02, rathjen06b} and with Aczel's Regular Extension Axiom.

\section{The Set Compactness Theorem}

The purpose of this section is to prove that, in {\bf CZF}, the combination of {\bf (WS)} and {\bf (AMC)} implies the Set Compactness Theorem. To state this Set Compactness Theorem, we need to review the basics of the theory of inductive definitions in {\bf CZF}, which will be our metatheory in this section.

\begin{defi}{notation}
If $X$ is a class, we will denote by ${\rm Pow}(X)$ the class of subsets of $X$ and if $X$ is a set, we will denote by ${\rm Surj}(X)$ the class of surjections onto $X$.
\end{defi}

\begin{defi}{inductivedef}
Let $S$ be a set. An \emph{inductive definition} on $S$ is a subset $\Phi$ of ${\rm Pow}(S) \times S$. If $\Phi$ is an inductive definition, then a subclass $A$ of $S$ is \emph{$\Phi$-closed}, if
\[ X \subseteq A \Rightarrow a \in A \]
whenever $(X, a)$ is in $\Phi$.
\end{defi}

Within {\bf CZF} one can prove that for every subclass $U$ of $S$ there is a least $\Phi$-closed subclass of $S$ containing $U$ (see \cite{aczelrathjen01}); it is denoted by $I(\Phi, U)$. The Set Compactness Theorem is the combination of the following two statements:
\begin{enumerate}
\item $I(\Phi, U)$ is a set whenever $U$ is.
\item There is a set $B$ of subsets of $S$ such that for each class $U \subseteq
S$ and each $a \in I(\Phi, U)$ there is a set $V \in B$ such that $V \subseteq U$
and $a \in I(\Phi, V)$.
\end{enumerate}
As said, the Set Compactness Theorem is not provable in {\bf CZF} proper, but we will show in this section that it becomes provable when we extend {\bf CZF} with {\bf (WS)} and {\bf (AMC)}.

To prove the result it will be convenient to introduce the notion of a \emph{collection square}. In the definition we write for any function $f: B \to A$ and each $a \in A$,
\[ B_a = f^{-1}(a) = \{ b \in B \, : \, f(b) = a \}, \]
as is customary in categorical logic.

\begin{defi}{collectionsquareinsets}
A commuting square in the category of sets
\diag{ D \ar[r]^{q} \ar[d]_g & B \ar[d]^f \\
C \ar[r]_{p} & A }
will be called a \emph{collection square}, if
\begin{enumerate}
\item the inscribed map $D \to B \times_A C$ is a surjection (meaning that for each pair of elements $b \in B, c \in C$ with $f(b) = p(c)$ there is at least one $d \in D$ with $q(d) = b$ and $g(d) = c$),
\item and for each $a \in A$ and each surjection $e: E \epi B_a$
there is a $c \in p^{-1}(a)$ and a map $h: D_c \to E$ such that the triangle
\diag{ & E \ar@{->>}[dr]^e & \\
D_c \ar[ur]^h \ar@{->>}[rr]_{q \upharpoonright D_c} & & B_a }
commutes.
\end{enumerate}
Note that the last condition implies in particular that $p$ must be a surjection.
\end{defi}

Observe that ${\bf (AMC)}$ can be rephrased as: any map $f: X \to 1$ fits into a collection square
\diag{ Y \ar[r]^q \ar[d]_g & X \ar[d]^f \\
I \ar[r]_p & 1. }
In fact, {\bf (AMC)} implies that this property holds for any map.
\begin{prop}{AMCimpliesenoughcollsquares}
{\bf (AMC)} implies that any function \func{f}{B}{A} fits into a collection square
\diag{ D \ar[r]^q \ar[d]_g & B \ar[d]^f \\
C \ar[r]_p & A. }
\end{prop}

Note that for the strong version of {\bf (AMC)}, this is really Proposition 4.6 in \cite{moerdijkpalmgren02}.

\begin{proof} Let $f: B \to A$ be any function. {\bf (AMC)} implies that:
\[ (\forall a \in A) \, (\exists \alpha \in {\rm Pow}({\rm Surj}(B_a))) \, \mbox{every surjection onto }
B_a \mbox{ is refined by one in } \alpha. \]
We may now apply the collection axiom to this statement: this gives us a surjection $p: C \epi A$ together with, for every $c \in C$, an inhabited collection $Z_c \subseteq {\rm Pow}({\rm Surj}(B_a))$ such that:
\[ (\forall c \in C) \, (\forall \alpha \in Z_c) \, \mbox{every cover of } B_{p(c)} \mbox{ is refined by an element of } \alpha. \]
Let $T_c = \bigcup Z_c$. Then clearly:
\[ (\forall c \in C) \, \mbox{every surjection onto } B_{p(c)} \mbox{ is refined by an element of } T_c. \]
So set $D = \{ (c \in C, t \in T_c, x \in {\rm dom}(t)) \}$ and let $g$ be the projection on the first coordinate and $q(c, t, x) = t(x)$. All the required verifications are now very easy and left to the reader.
\end{proof}

\begin{theo}{setcompactness} The combination of {\bf (WS)} and {\bf (AMC)} implies the Set Compactness Theorem.
\end{theo}
\begin{proof}
Let $S$ be a set and $\Phi$ be an inductive definition on $S$. Our aim is to construct a set $B$ of subsets of $S$ such that for each class $U \subseteq
S$ and each $a \in I(\Phi, U)$ there is a set $V \in B$ such that $V \subseteq U$
and $a \in I(\Phi, V)$.

Write $\Psi = \{ (X, a, b) \, : \, (X, a) \in \Phi, b \in X \}$ and consider the map $h: \Psi \to \Phi$ given by projection onto the first two coordinates. By composing this map with the sum inclusion $\Phi \to \Phi + S$, we obtain a map we call $f$.

{\bf (AMC)} implies that $f$ fits into a collection square with a small map $g$ on the left, as in:
\diaglab{collsqSC}{ D \ar[r]^q \ar[d]_g & \Psi \ar[d]^{f} \ar[r]^{\pi_3} & S \\
C \ar[r]_(.4)p & \Phi + S. }
We take the W-type $W(g)$ associated to $g$ and, because {\bf (WS)} holds, $W(g)$ is a set. We wish to regard certain elements of $W(g)$ as \emph{proofs}.

To identify these, define a map ${\rm conc}: W(g) \to S$ assigning to every element of $W(g)$ its \emph{conclusion} by case distinction, as follows:
\begin{eqnarray*}
{\rm conc}({\rm sup}_c(t)) & = & \left\{ \begin{array}{ll}
p(c) & \mbox{if } p(c) \in S, \\
a & \mbox{if } p(c) = (X, a) \in \Phi.
\end{array} \right.
\end{eqnarray*}
In addition, define inductively the function ${\rm ass}: W(g) \to {\rm Pow}(S)$ assigning to
every element of $W(g)$ its \emph{set of assumptions} as follows:
\begin{eqnarray*}
{\rm ass}({\rm sup}_c(t)) & = & \left\{ \begin{array}{ll}
\{ p(c) \} & \mbox{if } p(c) \in S, \\
\bigcup_{d \in g^{-1}(c)} {\rm ass}(td) & \mbox{otherwise.}
\end{array} \right.
\end{eqnarray*}
Finally, call an element ${\rm sup}_c(t) \in W(g)$ \emph{well-formed}, if $p(c) =
(X, a) \in \Phi$ implies that for all $d \in D_c$ the conclusion of $t(d)$ is
$\pi_3q(d)$ (the map $\pi_3q$ is the composite along the top in diagram \refdiag{collsqSC}); call it a \emph{proof}, if it and all its subtrees are well-formed.
Because the collection of subtrees of some tree in a W-type is a set (see the proof of Theorem 6.13 in \cite{bergmoerdijk08}), the collection of proofs is a set by bounded separation.

The proof will be finished once we show that:
\begin{eqnarray*}
 I(\Phi, U) & = & \{ x \in S \, : \, \mbox{there is a proof all whose assumptions
 belong to $U$} \\ & & \mbox{and whose conclusion is } x \}.
\end{eqnarray*}
Because from this expression it follows by bounded separation that $I(\Phi, U)$ is a set whenever $U$ is; in addition, it implies that the set $B = \{ {\rm ass}(w) \, : \, w \in W(g) \}$ is as required by the second half of the Set Compactness Theorem.

In other words, we have to show that
\begin{eqnarray*}
 J(\Phi, U) & = & \{ x \in S \, : \, (\exists w \in W(g)) \, w \mbox{ is a proof, }
 {\rm ass}(w) \subseteq U \mbox{ and } {\rm conc}(w) = x \}.
\end{eqnarray*}
is $\Phi$-closed, contains $U$ and is contained in every $\Phi$-closed subclass of
$S$ which contains $U$. To see that $J(\Phi, U)$ contains $U$, note that an element
${\rm sup}_c(t)$ with $p(c) = s \in S$ and $t$ the empty function is a proof whose sole assumption is $s$ and
whose conclusion is $s$. To see that it is $\Phi$-closed, let $(X, a) \in \Phi$ and
suppose that
\[ (\forall b \in X) \, b \in J(\Phi, U); \]
in other words, that
\[ (\forall b \in X) \, (\exists w \in W(g)) \, w \mbox{ is a proof, } {\rm ass}(w)
\subseteq U \mbox{ and } {\rm conc}(w) = x. \]
Now we use the collection square property to obtain a $c \in C$ with $p(c) = (X, a)
\in \Phi$ and a map $t: D_c \to W(g)$ such that for all $d \in D_c$, $td$ is a proof
with ${\rm ass}(td) \subseteq U \mbox{ and } {\rm conc}(td) = qd$. Hence ${\rm
sup}_c(t)$ is a proof with assumptions contained in $U$ and conclusion $a$ and
therefore $a \in J(\Phi, U)$, as desired.

It remains to show that $J(\Phi, U)$ contains every $\Phi$-closed subclass $A$
containing $U$. To this purpose, we prove the following statement by induction:
\begin{quote}
For all $w \in W(g)$, if $w$ is a proof and ${\rm ass}(w) \subseteq U$, then ${\rm
conc}(w) \in A$.
\end{quote}
So let $w = {\rm sup}_c(t) \in W(g)$ be a proof such that ${\rm ass}(w) \subseteq
U$. For every $d \in D_c$, $td$ is a proof with ${\rm ass}(td) \subseteq U$, so we
have ${\rm conc}(td) \in A$ by induction hypothesis. Now we make a case distinction as to whether $p(c)$ belongs to $S$ or $\Phi$:
\begin{itemize}
\item If $p(c) \in S$, then $w$ is a proof whose sole assumption is
$pc$ and whose conclusion is $pc$. Then it follows from ${\rm ass}(w) \subseteq U$ that $pc \in U \subseteq A$. Hence ${\rm conc}(w) = pc \in A$, as desired.
\item In case $p(c) = (X, a) \in \Phi$, we have to show $a = {\rm conc}(w) \in A$ and for that it suffices
to show that $b \in A$ for all $b \in X$, since $A$ is $\Phi$-closed. But for every
$b \in X$, there is a $d \in D_c$ with $p(d) = (X, a, b)$ and, since $w$ is
well-formed, $b = {\rm conc}(td) \in A$.
\end{itemize}
This completes the proof.
\end{proof}

\section{Stability under exact completion}

In the following sections we will show that {\bf (AMC)} and {\bf (WS)} are stable under exact completion, realizability, presheaves and sheaves, respectively. We will do this in the setting of algebraic set theory as developed in our papers \cite{bergmoerdijk08, bergmoerdijk09, bergmoerdijk10b} and to that purpose, we reformulate {\bf (AMC)} in categorical terms.

\begin{defi}{collectionsquare}
We call a square
\diag{ D \ar[r]^{q} \ar[d]_g & B \ar[d]^f \\
C \ar[r]_{p} & A }
a \emph{covering square}, if both $p$ and the canonical map $D \to B \times_A C$ are covers. We will call it a \emph{collection square}, if, in addition, the following
statement holds in the internal logic: for all $a \in A$ and covers $e: E \epi B_a$
there is a $c \in p^{-1}(a)$ and a map $h: D_c \to E$ such that the triangle
\diag{ & E \ar[dr]^e & \\
D_c \ar[ur]^h \ar[rr]_{q_c = q \upharpoonright D_c} & & B_a }
commutes. Diagrammatically, one can express the second condition by asking that any
map $X \to A$ and any cover $E \epi X \times_A B$ of the pullback fit into a cube
\begin{displaymath}
\xymatrix@!0{
 & Y \times_C D  \ar[rr] \ar'[d][dd] \ar[dl] &  & E \ar@{->>}[rr] & & X \times_A B
 \ar[dd] \ar[dl] \\
D \ar@{->>}[rrrr] \ar[dd] & &  &  & B \ar[dd]  \\
 & Y \ar@{->>}'[rrr][rrrr] \ar[dl] & & & & X, \ar[dl] \\
C \ar@{->>}[rrrr] &  & & & A }
\end{displaymath}
such that the face on the left is a pullback as indicated and the face at the back is covering.
\end{defi}

\begin{rema}{asymmetry} Note that being a covering or a collection square is really a property of an oriented square and in the definition the maps $p$ and $f$ play different roles. In this paper we will always draw collection and covering squares in such a way that the property holds from ``left to right'' (as in the definition), instead of from ``top to bottom''.
\end{rema}

In categorical terms the axiom now reads:
\begin{description}
\item[Axiom of Multiple Choice (AMC):]For any small map \func{f}{Y}{X}, there is a
cover \func{q}{A}{X} such that $q^*f$ fits into a collection square in which all
maps are small:
\diag{ D \ar[d] \ar@{->>}[r] &  A \times_X Y \ar[d]^{q^*f} \ar@{->>}[r] & Y \ar[d]^f
\\
C \ar@{->>}[r] & A \ar@{->>}[r]_q & X.}
\end{description}

We now proceed to show this axiom is stable under exact completion. We work in the setting of \cite{bergmoerdijk08} and use the same notation and terminology.  In particular, $(\ct{E}, \smallmap{S})$ will be a category with display maps and $(\overline{\ct{E}}, \overline{\smallmap{S}})$ will be its bounded exact completion as discussed in \cite{bergmoerdijk08}. If we say that {\bf (AMC)} holds in $(\ct{E}, \smallmap{S})$, then we will mean that {\bf (AMC)} holds with the phrase ``small map'' replaced by ``display map''.

We begin by stating two lemmas about collection squares:

\begin{lemm}{AMCexcompllemma1}
The embedding ${\bf y}: \ct{E} \to \overline{\ct{E}}$ preserves and reflects collection squares.
\end{lemm}
\begin{proof}
Recall from Theorem 5.2 in \cite{bergmoerdijk08} that ${\bf y}$ has the following properties:
\begin{enumerate}
\item ${\bf y}$ is full and faithful,
\item ${\bf y}$ is covering, i.e., every object in $\overline{\ct{E}}$ is covered by
one in the image of {\bf y},
\item ${\bf y}$ preserves and reflects pullbacks,
\item ${\bf y}$ preserves and reflects covers.
\end{enumerate}
From items 3 and 4 it follows that ${\bf y}$ preserves and reflects covering squares.

To show that ${\bf y}$ preserves collection squares, suppose that we have a
collection square
\diag{ D \ar[r]^{\sigma} \ar[d] & B \ar[d] \\
C \ar[r]_{\rho} & A }
in \ct{E}, a map $X \to {\bf y}A$ and a cover $E \twoheadrightarrow {\bf y}B
\times_{{\bf y}A} X$. Using item 2, we find a cover $q: {\bf y}X' \to X$ and a cover
${\bf y}E' \to (\id_{{\bf y}B} \times_{{\bf y}A} q)^*E$. Then we may apply the
collection square property in $\ct{E}$ to obtain a diagram of the desired shape.

To see that ${\bf y}$ reflects collection squares, suppose that
\diag{ {\bf y}D \ar[r]^{{\bf y}\sigma} \ar[d]_{{\bf y}g} & {\bf y}B \ar[d] \\
{\bf y}C \ar[r]_{{\bf y}\rho} & {\bf y}A }
is a collection square in $\overline{\ct{E}}$.  Then, if $X \to A$ is any map and $E \to B \times_A X$ is a cover in $\ct{E}$, this is preserved by ${\bf y}$, so that we obtain a map $t: Y \to {\bf y}C$ and a covering square
\diag{ q^*{\bf y}D \ar[r] \ar[d]_{t^*{\bf y}g} & E \ar@{->>}[r] & {\bf y}B \times_{{\bf y}A} {\bf y}X \ar[d] \\
Y \ar@{->>}[rr] & & {\bf y}X.}
By covering $Y$ with an object ${\bf y}Y'$, sticking the pullback square
\diag{ {\bf y}(D \times_C Y') \ar[d] \ar[r] & q^*{\bf y}D \ar[d] \\
{\bf y}Y' \ar@{->>}[r] & Y }
to the left of the previous diagram and reflecting back along ${\bf y}$, we obtain a diagram of the desired form in $\ct{E}$.
\end{proof}

\begin{lemm}{AMCexcompllemma2}
Suppose we have a commuting diagram of the following shape
\diag{ F \ar@{->>}[r]^\beta \ar[d] & D \ar@{->>}[r]^{\sigma} \ar[d] & B \ar[d] \\
E \ar@{->>}[r]_\alpha & C \ar@{->>}[r]_{\rho} & A, }
where both squares are covering. If one of the two inner squares is a collection
square, then so is the outer square.
\end{lemm}
\begin{proof}
Covering squares compose (Lemma 2.4.2 in \cite{bergmoerdijk08}), so the outer square
is covering. From now on, we reason in the internal logic. Assume that left square
is a collection square. Suppose $a \in A$ and $q: T \epi B_a$. Since $\rho$ is
a cover, we find a $c \in C$ such that $\rho(c) = a$, and because the square on the
left is collection, we find an element $e \in E$ together with a map $p: F_e \to
\sigma_c^* T$ such that the following diagram commutes:
\diag{ & \sigma_c^*T \ar[r] \ar@{->>}[d]^{\sigma_c^* q } & T \ar@{->>}[d]^q \\
F_e \ar[ur]^p \ar[r]_{\beta_e} & D_c \ar[r]_{\sigma_c} & B_a. }
Since $(\sigma\beta)_e = \sigma_c\beta_e$, this yields the desired result. The case
where the right square is a collection square is very similar, but easier.
\end{proof}

\begin{rema}{AMCanders}
Note that it follows from this lemma that {\bf (AMC)} could also have been formulated as
follows: every small map $f$ is covered by a small map $f'$ which is the right
edge in a collection square in which all maps are small (the same is true for display maps, see Lemma 2.11 in \cite{bergmoerdijk08}).
\end{rema}

\begin{theo}{AMCstableexcompl}
The axiom of multiple choice {\bf (AMC)} is stable under exact completion.
\end{theo}
\begin{proof}
Suppose that {\bf (AMC)} holds in \ct{E} and $f: B \to A$ is a small map in
$\overline{\ct{E}}$. By definition this means that $f$ is covered by a map of the form ${\bf y}f'$ with $f'$ display in \ct{E}. Since $f'$ is display in \ct{E} and {\bf (AMC)} holds in \ct{E}, we may cover $f'$ by a map $f''$ in \ct{E} which fits in a collection square in which all maps are display. That the same holds for $f$ in $\overline{\ct{E}}$ now follows from \reflemm{AMCexcompllemma1} and \refrema{AMCanders}.
\end{proof}

We will now show that {\bf (AMC)} is also reflected by exact completions.

\begin{lemm}{exclemma3}
Suppose we are given a commuting cube of the form
\begin{displaymath}
\xymatrix@!0{
 & D  \ar[rr] \ar'[d][dd] &  & B \ar[dd]  \\
H \ar[rr] \ar[dd] \ar@{->>}[ur] &  & F \ar[dd] \ar@{->>}[ur]  \\
 & C \ar'[r][rr] & & A,  \\
G \ar[rr] \ar@{->>}[ur] & & E \ar@{->>}[ur] }
\end{displaymath}
in which the right face is a pullback and the faces at the bottom and on the left
are covering. If the face at the back is a collection square, then so is the face at the front.
\end{lemm}
\begin{proof}
An easy argument using the internal logic.
\end{proof}

\begin{theo}{AMCreflectedexcompl}
The axiom of multiple choice {\bf (AMC)} is reflected by exact completions.
\end{theo}
\begin{proof} Suppose that {\bf (AMC)} holds in $\overline{\ct{E}}$ and $f: Y \to X$ is a small map in
\ct{E}. This means that there is a cover $q: D \to {\bf y}X$ in $\overline{\ct{E}}$
such that $q^*({\bf y}f)$ fits into the right hand side of a collection square in
which all maps are small. We construct a diagram
\begin{displaymath}
\xymatrix@!0{
 & A  \ar@{->>}[rr] \ar'[d][dd] &  & B \ar[dd]^{q^*({\bf y}f)} \ar@{->>}[rr] & &
 {\bf y}Y \ar[dd]^{{\bf y}f}  \\
{\bf y}E \ar@{->>}[rr] \ar[dd] \ar@{->>}[ur] &  & {\bf y}F \ar[dd] \ar@{->>}[ur]
\\
 & C \ar@{->>}'[r][rr] & & D \ar@{->>}[rr]_q & & {\bf y}X,  \\
{\bf y}G \ar@{->>}[rr] \ar@{->>}[ur] & & {\bf y}H \ar@{->>}[ur]_p }
\end{displaymath}
as follows. First we construct the bottom and left faces using Lemma 5.6 in
\cite{bergmoerdijk08}, so that both are covering and the maps ${\bf y}E \to {\bf
y}G$ and ${\bf y}G \to {\bf y}H$ are both small. Then the right face is constructed
by pullback and since ${\bf y}$ preserves pullbacks, we may assume that the result
is an object of the form ${\bf y}F$ and the map ${\bf y}F \to {\bf y}H$ is
($qp)^*({\bf y}f)$; in particular, it is small. To finish the construction of the
cube, we have to find a map ${\bf y}E \to {\bf y}F$: but that we obtain from the
universal property of ${\bf y}F$. Note that it follows from Lemma 2.11 in
\cite{bergmoerdijk08} that this map is also small. By the previous lemma, we now
know that the front face of the cube is a collection square in which all maps are
small. Since collection squares, small maps and pullbacks are reflected by {\bf y},
we have shown that {\bf (AMC)} is reflected by exact completion.
\end{proof}

Note that in \cite{bergmoerdijk08} we were unable to show that the axioms {\bf
($\Pi$S)} and {\bf (WS)} are stable under exact completion. In the presence of {\bf (AMC)}, however, we can.

\begin{theo}{PiSstableexcompl}
In the presence of {\bf (AMC)}, the exponentiation axiom {\bf ($\Pi$S)} is stable under exact completion.
\end{theo}
\begin{proof}
Relative to {\bf (AMC)} the exponentiation axiom is equivalent to fullness
(see \cite[Proposition 2.16]{bergmoerdijk10b}), so this follows from the stability of the fullness axiom under exact completion (Proposition 6.25 in \cite{bergmoerdijk08}).
\end{proof}

\begin{theo}{WSstableexcompl}
In the presence of {\bf (AMC)}, the axiom {\bf (WS)} is preserved by exact completion. If $(\ct{E}, \smallmap{S})$ is a category with small maps, it will also be reflected.
\end{theo}
\begin{proof}
The proof of Theorem 6.18 in
\cite{bergmoerdijk08} implies that the functor ${\bf y}$ preserves W-types. It also preserves smallness and if $\smallmap{S}$ is a class of small maps, it will reflect smallness as well (see \cite[Theorem 5.2]{bergmoerdijk08}).  Hence it follows that exact completions of categories with small maps reflect {\bf (WS)}.

It also follows that W-types for maps of the form ${\bf y}g$ with
$g$ a display map in \ct{E} are small in $\overline{\ct{E}}$. The proof of the
stability of {\bf (AMC)} under exact completion implies that for every small map $f: B
\to A$ in $\overline{\ct{E}}$ there is a cover $q: A' \epi A$ such that $q^*f$ fits
into a collection square with such a map ${\bf y}g$ on the left. It is a consequence
of the proof of Proposition 6.16 in \cite{bergmoerdijk08} that the W-type associated
to $q^*f$ is small and a consequence of Proposition 4.4 in \cite{moerdijkpalmgren00}
that the W-type associated to $f$ is small.
\end{proof}

\section{Stability under realizability}

In this section we show that the axiom of multiple choice is stable under
realizability. Recall from \cite{bergmoerdijk09} that the realizability category
over a predicative category of small maps \ct{E} is constructed as the exact
completion of the category of \emph{assemblies}. Within the category of assemblies
we identified a class of maps, which was not quite a class of small maps. In a
predicative setting the correct description of these \emph{display maps} (as we
called them) is a bit involved, but for the full subcategory of \emph{partitioned
assemblies} the description is quite simple: a map of partitioned assemblies $f: (B,
\beta) \to (A, \alpha)$ is small, if the underlying map $f: B \to A$ in \ct{E} is small. Many questions about assemblies can be reduced to (simpler) questions about the
partitioned assemblies: essentially this is because the inclusion of partitioned
assemblies in assemblies is full, preserves finite limits and is \emph{covering}
(i.e., every assembly is covered by a partitioned assembly). Moreover, every display
map between assemblies is covered by a display map between partitioned assemblies.
For more details, we refer to \cite{bergmoerdijk09}.

\begin{theo}{AMCstablerealizability}
The axiom of multiple choice {\bf (AMC)} is stable under realizability.
\end{theo}
\begin{proof}
We show that {\bf (AMC)} holds in the category of assemblies over a predicative category of
classes \ct{E}, provided that it holds in \ct{E}. The result will then follow from
\reftheo{AMCstableexcompl} above.

Suppose $f$ is a display map of assemblies. We want to show that $f$ is covered by a
map which fits into a collection square in which all maps are display. Without loss of
generality, we may assume that $f$ is a display map of partitioned assemblies $(B,
\beta) \to (A, \alpha)$. For such a map, the underlying map $f$ in \ct{E} is small.
We may therefore use the axiom of multiple choice in \ct{E} to obtain a diagram of
the form
\diag{ F \ar@{->>}[r]^q \ar[d] & D \ar@{->>}[r]^{s} \ar[d] & B \ar[d]^f \\
E \ar@{->>}[r]_p & C \ar@{->>}[r]_{r} & A, }
in which the square on the left is a collection square in which all maps are small
and the one on the right is a covering square. We obtain a similar diagram in the
category of (partitioned) assemblies
\diag{ (F, \phi) \ar@{->>}[r]^q \ar[d] & (D, \delta) \ar@{->>}[r]^s \ar[d] & (B,
\beta) \ar[d]^f \\
(E, \epsilon) \ar@{->>}[r]_p & (C, \gamma) \ar@{->>}[r]_r & (A, \alpha), }
by defining $\gamma: C \to \NN$ by $\gamma(c) = \alpha r(c)$, and similarly $\epsilon(e) = \alpha rp(c), \delta(d) = \beta
s(d)$ and $\phi(f) = \beta sq(f)$. It is clear that both squares are covering, so it
remains to check that the one on the left is a collection square.

So suppose we have a map $t: (X, \chi) \to (C, \gamma)$ and a cover
\[ h: (M, \mu) \epi (X, \chi) \times_{(C, \gamma)} (D, \delta) = (X \times_C D, \kappa) \]
in the category of assemblies. Without loss of generality, we may assume that both $(X,
\chi)$ and $(M, \mu)$ are partitioned assemblies and $(X \times_C D, \kappa)$ is the partitioned assembly with $\kappa(x, d) = < \chi(x), \delta(d) >$. Define
\begin{eqnarray*}
X'& = & \{ (x \in X, n \in \NN) \, : \, n \mbox{ realizes the surjectivity of } h
\}, \\
M'& = & \{ (m \in M, n \in \NN) \, : \, (\pi_1h(m), n) \in X' \mbox{ and } n \cdot
\kappa(h(m)) = \mu(m) \},
\end{eqnarray*}
and consider the diagram
\diag{ M' \ar[r]^(.4){h'} & X'\times_C D \ar[r] \ar[d] & D \ar[d] \\
& X'\ar[r]_{t\pi_0} & C }
with $h'(m, n) = (\pi_1h(m), n, \pi_2h(m))$. By definition of $X'$, the map $h'$ is
a cover, so we may apply the collection square property in \ct{E} to obtain a map
$w: Y \to E$ and a covering square of the form
\diag{ w^*F \ar[d]_{l} \ar[r]^{k'} & M' \ar[r]^{h'} & X'\times_C D \ar[d] \\
Y \ar[rr]_v & & X'. }
Writing $u = \pi_1v: Y \to X$ and $\upsilon(y) = < \chi u(y), \pi_2 v(y) >$, we obtain
a similar covering diagram
\diag{ w^*(F, \phi) \ar[d]_l \ar[r]^k & (M, \mu) \ar[r]^(.4){h} & (X \times_C D, \kappa)
\ar[d] \\
(Y, \upsilon) \ar[rr]_u & & (X, \chi) }
in the category of assemblies:
\begin{enumerate}
\item The map $u$ is a cover, essentially because $\pi_1: X' \to X$ is.
\item The map $k = \pi_0k'$ is tracked, because the realizer of an element $z$ in
$w^*(F, \phi)$ is the pairing of the realizers of its images $(hk)(z)$ and $l(z)$.
From the latter, one can compute the second
component $n$ of $(vl)(z)$. One may now compute the realizer of $k(z)$ by applying
this $n$ to the realizer of $(hk)(z)$ (by definition of $M'$).
\item The square is a quasi-pullback, with the surjectivity of the unique map to the
pullback being realized by the identity.
\end{enumerate}
This concludes the proof.
\end{proof}

Again, we are able to show that {\bf (AMC)} is also reflected by realizability.

\begin{theo}{AMCreflunderrealiz}
The axiom of multiple choice is reflected by realizability.
\end{theo}
\begin{proof}
By \reftheo{AMCreflectedexcompl} it is sufficient to prove that the axiom of multiple choice holds in \ct{E} whenever it holds in$ \Asm(\ct{E})$. Recall that there are two functors $\nabla: \ct{E} \to \Asm(\ct{E})$ and $\Gamma:
\Asm(\ct{E}) \to \ct{E}$, with $\nabla$ sending an object $X$ to the pair $(X, X
\times \NN)$ and $\Gamma$ sending an object $(X, \chi)$ to $X$. Both $\nabla$ and
$\Gamma$ preserve small maps, pullbacks, covers and (hence) covering squares. In
addition, $\Gamma \nabla = \id$, so it suffices to show that $\Gamma$ preserves
collection squares.

Let
\diag{ (D, \delta) \ar[r] \ar[d] & (C, \gamma) \ar[d] \\
(B, \beta) \ar[r] & (A, \alpha) }
be a collection square in $\Asm(\ct{E})$ and suppose we are given a map $f: X \to A$
and a cover $q: E \to X \times_A B$ in \ct{E}. By putting $\chi(x) = \alpha(fx)$,
$\psi(x, b) = \{ < m, n > \, : \, m \in \chi(x), n \in \beta(b) \}$ and $\epsilon(e) = \psi(q(e))$ and using the
collection square property of the figure above, we obtain a diagram in
$\Asm(\ct{E})$ of the following form:
\begin{displaymath}
\xymatrix@!0{
 & & (Y \times_C D, \rho)  \ar[rrr] \ar'[d][dd] \ar[dll] &  & & (E, \epsilon)
 \ar@{->>}[rrr] & & & (X \times_A B, \psi) \ar[dd] \ar[dll] \\
(D, \delta) \ar@{->>}[rrrrrr] \ar[dd] & & & & &  & (B, \beta) \ar[dd]  \\
 & & (Y, \omega) \ar@{->>}'[rrrr][rrrrrr] \ar[dll] & & & & & & (X, \chi), \ar[dll]
 \\
(C, \gamma) \ar@{->>}[rrrrrr] & & &  & & & (A, \alpha) },
\end{displaymath}
where the left face is a pullback and the face at the back is covering. By applying
$\Gamma$ to this diagram we obtain the desired result.
\end{proof}

In \cite{bergmoerdijk09} we were unable to show that the axioms {\bf ($\Pi$S)} and
{\bf (WS)} are stable under realizability. This was because we were unable to show
that they were stable under exact completion. But as that was our only obstacle, we
now have:

\begin{theo}{AMCimplaxiomsstablereal}
In the presence of {\bf (AMC)}, the axioms {\bf ($\Pi$S)} and {\bf (WS)} are stable under realizability.
\end{theo}
\begin{proof}
Since both {\bf ($\Pi$S)} and {\bf (WS)} are inherited by the category of assemblies
(Propositions 20 and 21 in \cite{bergmoerdijk09}), this follows from
\reftheo{PiSstableexcompl} and \reftheo{WSstableexcompl}, respectively.
\end{proof}

\begin{theo}{AMCimplWSeflectedreal}
In the presence of {\bf (AMC)}, the axiom {\bf (WS)} is reflected by realizability.
\end{theo}
\begin{proof}
It follows immediately from the description of W-types in the category of assemblies (see \cite[Proposition 21]{bergmoerdijk09}) that $\nabla$ preserves W-types. As $\nabla$ reflects smallness, it follows that the axiom {\bf (WS)} is reflected by realizability.
\end{proof}

\section{Stability under sheaves}

In this section we will show that {\bf (AMC)} is preserved and reflected by sheaf extensions. Theorem 4.21 in \cite{bergmoerdijk10b} shows that {\bf (WS)} is preserved by sheaf extensions in the presence of strong {\bf (AMC)}, but it is not hard to see that the same argument shows that {\bf (WS)} is preserved with our present version of {\bf (AMC)}; that it is also reflected will be \reftheo{Wtypesreflsheaves} below. We will use notation and terminology from \cite{bergmoerdijk10b}. In particular, $(\ct{E}, \smallmap{S})$ is a predicative category with small maps satisfying the fullness axiom {\bf (F)} and $(\ct{C}, {\rm Cov})$ is an internal site in \ct{E} which has a presentation and whose codomain map ${\rm cod}: \ct{C}_1 \to \ct{C}_0$ is small. We will write $\pi^*: \pshct{\ct{E}}{\ct{C}} \to \ct{E}/\ct{C}_0$ for the forgetful functor and $\pi_!$ for its left adjoint, which sends a pair $(X, \sigma_X: X \to \ct{C}_0)$ to the following sum of representables:
\[ \pi_!(X, \sigma_X) = \sum_{x \in X} {\bf y}(\sigma_X(x)). \]
In other words,
\[ \pi_!(X, \sigma_X)(c) = \{ \, (x, \gamma) \, : \, d \in \ct{C}_0, x \in X(d), \gamma: c \to d \in \ct{C}_1 \, \}. \]
Given two objects $(Y, \sigma_Y)$ and $(X, \sigma_X)$ in $\ct{E}/\ct{C}_0$ and a pair of maps $k: Y \to X$ and $\kappa: Y \to \ct{C}_1$ such that
\diag{ Y \ar[d]_k \ar[r]_\kappa \ar@/^1pc/[rr]^{\sigma_Y} & \ct{C}_1 \ar[d]^{\rm cod} \ar[r]_{\rm dom} & \ct{C}_0 \\
X \ar[r]_{\sigma_X} & \ct{C}_0}
commutes, we obtain a map of presheaves $(k, \kappa)_!: \pi_!Y \to \pi_!X$ sending a pair $(y, \gamma)$ to $(k(y), \kappa_y \gamma)$. In fact, every map $\pi_! Y \to \pi_! X$ is of this form. Finally, we will write $i^*$ for the sheafification functor $\pshct{\ct{E}}{\ct{C}} \to \shct{\ct{E}}{\ct{C}}$ and $\rho_!: \ct{E}/\ct{C}_0 \to \shct{\ct{E}}{\ct{C}}$ for the composition of $\pi_!$ and $i^*$.

\begin{theo}{AMCstableundersheaves}
The axiom of multiple choice {\bf (AMC)} is preserved by sheaf extensions.
\end{theo}

Note that for strong {\bf (AMC)} this was proved in Section 10 of \cite{moerdijkpalmgren02}.

\begin{proof}
In this proof we assume that the underlying category \ct{C} has chosen pullbacks, something we may do without loss of generality. Consider a map
\[ i^*(k, \kappa)_!: \rho_! Y \to \rho_! X\] of sheaves in which $k$ is small. Since by definition every small map is covered by one of this form, it suffices to show that for every such map there is a cover such that pulling back the map along that cover gives a map which is the right edge in a collection square in which all maps are small.

Using {\bf (AMC)} in \ct{E}, we know that there is a cover $p: I \to X$ in \ct{E} such that $p^*k$ fits into a collection square in which all maps are small:
\diag{ B \ar[d]_m \ar@{->>}[r]^w & J \ar[d]^{p^*k} \ar@{->>}[r]^q & Y \ar[d]^{k}
\\
A \ar@{->>}[r]_v & I \ar@{->>}[r]_p & X. }
Now we make a host of definitions. Define $\sigma_I = \sigma_X p, \sigma_A = \sigma_X pv, \sigma_J = \sigma_Y q, \sigma_B = \sigma_Y qw, \mu_b = \kappa_{qwb}$. Furthermore, we define an object $S$ fibred over $A$: $S_a$ consists of pairs $(\gamma, \varphi)$ with $\gamma$ a map in $\ct{C}_1$ with codomain $\sigma_A(a)$ and $\varphi$ a map assigning to every $b \in B_a$ a sieve $S \in {\rm BCov}(\gamma^*\sigma_B(b))$, where $\gamma^*\sigma_B(b)$ denotes the following pullback in $\ct{C}$:
\diag{ \gamma^*\sigma_B(b) \ar[r] \ar[d] & \sigma_B(b) \ar[d]^{\mu_b} \\
\bullet \ar[r]_{\gamma} & \sigma_A(a). }
We also define an object $M$ fibred over $S$, with the fibre over $(a, \gamma, \varphi)$ consisting of pairs $b \in B_a$ and $\alpha \in \varphi(b)$. We obtain a commuting square as follows:
\diag{ M \ar[r]^h \ar[d]_n & B \ar[d]^m \\
S \ar@{->>}[r]_g & A, }
in which all maps are small and $g$ is a cover.

We apply {\bf (AMC)} again, but now to $n$. Strictly speaking, one would obtain a cover $r: P \to S$ such that $r^*n$ fits into the right-hand side of a collection square in which all maps are small. We claim that we may assume, without loss of generality, that $r = \id$, so that already $n$ fits into the right-hand side of a collection square. Its proof is a bit of a distraction from the main thread of the argument, so probably best skipped on a first reading.

{\bf Proof of the claim.} Applying {\bf (AMC)} to $n$ yields a diagram
\diag{ W \ar[d] \ar[r] & Q \ar[d] \ar[r] & M \ar[d]^n \\
V \ar[r] & P \ar@{->>}[r]_r & S }
in which the left square is collection and the right one a pullback. By applying the collection axiom to the small map $vg$ and the cover $r$, we obtain a diagram of the form
\diag{ S'' \ar[r] \ar[d] & S' \ar[r] \ar[d] & A' \ar[r] \ar[d] &
I' \ar[d] \\
P \ar[r]_r & S \ar[r]_g & A \ar[r]_v & I, }
in which $I' \to I$ is a cover, the two rightmost squares are pullbacks and $S'' \to S'$ is a small cover. The idea is to replace $J, A, B, S, M$ with their pullbacks along $I' \to I$. Call these $J', A', B', S', M'$, respectively. Crucially, $S'$ and $M'$ are then defined in the same way from $A'$ and $B'$ as $S$ and $M$ are defined from $A$ and $B$.

We pull back the collection square on $P$ along $S'' \to P$ and obtain a new collection square on $S''$, in which all maps are still small:
\begin{displaymath}
\xymatrix@!0{
 & W \ar'[d][dd]\ar[rrrr] & &  & & Q
 \ar[dd] \\
D \ar[rrrr] \ar[dd] \ar[ur] & &  &  & T \ar[dd] \ar[ur]  \\
 & V \ar'[rrr][rrrr] & & & & P \\
C \ar[rrrr] \ar[ur] &  & & & S''. \ar[ur] }
\end{displaymath}
By the universal property of $M'$ we obtain a map $T \to M'$ making the diagram
\begin{displaymath}
\xymatrix@!0{
& &  & Q \ar'[d][dd]\ar[rrrr] & &  & & M
 \ar[dd] \\
D \ar[rr] \ar[dd] & & T \ar[rrrr] \ar[dd] \ar[ur] & &  &  & M' \ar[dd] \ar[ur]  \\
  & & & P \ar'[rrr][rrrr] & & & & S \\
C \ar[rr] & & S'' \ar[rrrr] \ar[ur] &  & & & S'. \ar[ur] }
\end{displaymath}
commute. Note that this map is small, because all others in the front of the cube are. Since the left and right faces of the cube are pullbacks and the back is covering, the front of the cube is covering as well. Therefore not only the square on $S''$ is a collection square, but also the pasting of that square with the front of the cube (by \reflemm{AMCexcompllemma2}). As a result, we have
\diag{D \ar[r] \ar[d] & M' \ar[r] \ar[d] & B' \ar[d] \ar[r] & J' \ar[d] \ar[r] & Y \ar[d] \\
C \ar[r] & S' \ar[r] &  A' \ar[r] & I' \ar[r] & X, }
in which the first and third square (from the left) are collection squares in which all maps are small. This proves the claim. $\Box$

So from now on we work under the assumption that $r = \id$ and $n$ is the right-hand edge of a collection square. The result is a diagram of the following shape:
\diag{ D \ar[d]_s \ar[r]^f & M \ar[d]^n \ar[r]^h & B \ar[d]^m \ar[r]^w & J \ar[d]^{p^*k} \ar[r]^q & Y \ar[d]^k \\
C \ar[r]_e & S \ar[r]_g & A \ar[r]_v & I  \ar[r]_p & X, }
where the first and third square (from the left) are collection squares. Note that all maps in this diagram except for $p$ and $q$ are small. For convenience, we write $o = vge, t = whf$ and observe that $o$ is epi.

We wish to construct a diagram of the following shape in presheaves:
\diag{ \pi_! D \ar[d]_{(s, \sigma)_!} \ar[rr]^{(t, \theta)_!} & & \pi_! J \ar[rr]^{\pi_! q} \ar[d]^{(p^*k, \kappa q)_!}  & & \pi_! Y \ar[d]^{(k, \kappa)_!} \\
\pi_! C \ar[rr]_{(o, \omega)_!} & & \pi_! I \ar[rr]_{\pi_! p} & & \pi_! X.}
Understanding the right square should present no problems: but note that it is a pullback with a cover at the bottom. The remainder of the proof explains the left square and shows that its sheafification is a collection square in the category of sheaves. That would complete the proof.

Every element $c \in C$ determines an element $e(c) = (a, \gamma, \varphi) \in S$. We put $\omega_c = \gamma$ and $\sigma_C(c) = {\rm dom}(\gamma)$. Note that this turns $(o, \omega)_!$ into a cover. Similarly, every $d \in D$ determines an element $f(d) = (b, \gamma, \varphi, \alpha) \in M$. We put $\alpha_d = \alpha$, $\sigma_D(d) = {\rm dom}(\alpha)$, $\sigma_d = \pi_1 \circ \alpha$ and $\theta_d = \pi_2 \circ \alpha$, where $\pi_1$ and $\pi_2$ are the legs of the pullback square
\diag{ \gamma^* \sigma_B(b) \ar[d]_{\pi_1} \ar[r]^{\pi_2} & \sigma_B(b) \ar[d]^{\mu_b} \\
\bullet \ar[r]_\gamma & \sigma_A(mb).}
in \ct{C}. Note that this makes the map from $\pi_!D$ to the inscribed pullback of the left square locally surjective; hence its sheafification is covering.

In order to show that the sheafification of the left square is a collection square,
suppose that we have a map $z: V \to \rho_! I$ and a cover $d: Q \to z^* \rho_!
J$ of sheaves. Let $W$ be the pullback in presheaves of $V$ along $\pi_!(I) \to \rho_!(I)$ and cover $W$ using the counit $\pi_!\pi^*W \to W$. Writing $L = \pi^*W$, this means that we have a commuting square of presheaves
\diag{ \pi_!L \ar[rr]^{(r, \rho)_!} \ar[d] & & \pi_!(I) \ar[d] \\
V \ar[rr]_z & & \rho_!(I) }
in which the vertical arrows are locally surjective and the top arrow is of the form $(r, \rho)_!$. Finally, let $E: P \to \pi_!(L \times_I J)$ be the pullback of $d: Q \to z^* \rho_!J$ along the unique map $\pi_!(L \times_{I} J) \to z^*\rho_!J$ making
\begin{displaymath}
\xymatrix@!0{
  & P \ar[rr]^(.3)E \ar[dl] & &  \pi_! (L \times_I J) \ar[dl] \ar'[d][dd]\ar[rrrr] & &  & & \pi_! J
 \ar[dd] \ar[dl] \\
Q \ar[rr]_d & &  z^*\rho_! J \ar[rrrr] \ar[dd] & &  &  & \rho_! J \ar[dd]  \\
 & & & \pi_!L \ar'[rrr]^{(r, \rho)_!}[rrrr] \ar[dl] & & & & \pi_! I, \ar[dl] \\
& & V \ar@{->>}[rrrr]_{z} &  & & & \rho_!I }
\end{displaymath}
commute.

Since $d$ is locally surjective, the same applies to $E$. Reasoning in the internal logic, this means that the following statement holds:
\[ (\forall l \in L) \, (\forall j \in J_{r(l)}) \, (\exists S \in {\rm
BCov}(\rho_l^*\sigma_J(j)) \, (\forall \alpha \in S) \, (\exists p \in P) \, E(p)
= ((l, j), \alpha). \]
Using the collection square property, we find for every $l \in L$ an element $a \in A$ with $v(a) = r(l)$ together with a function $\varphi \in \Pi_{b \in B_a} {\rm BCov}(\rho_l^*(\sigma_B(b)))$ such that:
\[ (\forall b \in B_a) \, (\forall \alpha \in \varphi(b)) \, (\exists p \in P) \,
E(p) = ((l, w(b)), \alpha). \]
Again using the collection square property, we find for every $l \in L$ an element $c \in C$ with $e(c) = (a, \rho_l, \varphi)$ and a function $\psi: D_c \to P$ such that
\[ (\forall d \in D_c) \, E(\psi(d)) = ((l, t(d)), \alpha_d). \]
(Remember $t = whf$.) Therefore we obtain a diagram of the following shape in \ct{E}:
\begin{displaymath}
\xymatrix@!0{
 & U \times_C D \ar[dl] \ar'[d][dd]\ar[rrrr]^b & &  & & L \times_{I} J
 \ar[dd] \ar[dl] \\
D \ar@{->>}[rrrr] \ar[dd] & &  &  & J \ar[dd]  \\
 & U \ar@{->>}'[rrr]^\eta[rrrr] \ar[dl]_\epsilon & & & & L, \ar[dl]^r \\
C \ar@{->>}[rrrr]_{o} &  & & & I }
\end{displaymath}
with
\begin{eqnarray*}
U & = & \{ (l \in L, c \in C, \psi: D_c \to P) \, : \, o(c) = r(l), \omega_c = \rho_l
\mbox{ and } \\
& & (\forall d \in D_c) \, E(\psi(d)) = ((l, t(d)), \alpha_d) \},
\end{eqnarray*}
$\eta, \epsilon$ the obvious projections and $b(l, d, \psi) = (l, t(d))$. We now obtain a
diagram of presheaves of the shape
\begin{displaymath}
\xymatrix@!0{
 & & \pi_!(U \times_C D)  \ar[rrrrrr]^{(b, \beta)_!} \ar'[d][dd]
 \ar[dll] & &  & & & & \pi_!(L \times_{I} J) \ar[dd] \ar[dll] \\
\pi_!D \ar@{->>}[rrrrrr]^{(t, \theta)_!} \ar[dd]_{(s, \sigma)_!} & &  & & & & \pi_! (J)
\ar[dd]  \\
 & & \pi_!U \ar@{->>}'[rrrr]^{\pi_!\eta}[rrrrrr] \ar[dll]_{\pi_! \epsilon} & & & & & &
 \pi_!L, \ar[dll]^{(r, \rho)_!} \\
\pi_!C \ar@{->>}[rrrrrr]_{(o, \omega)_!} & & &  & & & \pi_!I }
\end{displaymath}
with $\sigma_U(c, l, \psi) = {\rm dom}(\omega_c) = \sigma_C(c) = {\rm dom}(\rho_l) = \sigma_L(l)$. In this diagram, the square on the left is a pullback square computed in the customary manner with $\sigma_{U \times_C D}(l, d, \psi) = \sigma_D(d)$, and the unique map $(b, \beta)_!$ filling the diagram is given by $\beta_{(l, d, \psi)} = \alpha_d$.

We now show that the sheafification of the square at the back is covering. First observe that $\pi_!\eta$ is a cover, since $\eta$ is. Therefore we only need to show that the square at the back is ``locally'' a quasi-pullback. To that end, suppose we have an element $((l, c, \psi), \pi_1)$ in $U$ and element $((l, j), \id) \in \pi_!(L \times_I J)$, where $\pi_1$ is the projection obtained as in
\diag{ \sigma_{L \times_I J}(l, j) \ar[d]_{\pi_1} \ar[rr]^{\pi_2} & & \sigma_J(j) \ar[d]^{\kappa_{qj}} \\
\bullet \ar[rr]_{\omega_c = \rho_l} & & \sigma_I((p^*k)(j)).}
If $e(c) = (a, \omega_c, \varphi)$, then we find a $b \in B_a$ with $w(b) = j$. Writing $T = \varphi(b) \in {\rm BCov}(\sigma_{L \times_I J}(l, j))$, we find for every $\alpha \in T$ an element $d \in D_c$ with $f(d) = (b, \omega_c, \varphi, \alpha)$. Projecting $((l, d, \psi), \id) \in \pi_!(U \times_C D)$ to $\pi_!(L \times_I J)$ yields $((l, t(d)), \alpha_d) = ((l, j), \id) \cdot \alpha$ and projecting $((l, d, \psi), \id) \in \pi_!(U \times_C D)$ to $\pi_!U$ yields $((l, c, \psi), \pi_1 \circ \alpha_d) = ((l, c, \psi), \pi_1) \cdot \alpha$. This shows that the square at the back is ``locally'' covering. (We have used here that every element in an object of the form $\pi_!Z$ is a restriction of one of the form $(z, \id)$ and that it therefore suffices for proving that a map $Q: R \to \pi_!Z$ is locally surjective to show that every element of the form $(z, \id)$ is ``locally hit'' by $Q$.)

To complete the proof we need to show that $(b, \beta)_!$ factors through $E: P \to \pi_!(L \times_I J)$. But to define a map $G: \pi_!(U \times_C D) \to P$ is, by the adjunction, the same thing as to give a map $U \times_C D \to P$, which we can do by sending $(l, d, \psi)$ to $\psi(d)$. To show that $(b, \beta)_! = E \circ G$, it suffices to calculate:
\begin{eqnarray*}
(E \circ G)((l, d, \psi), \id) & = & E(\psi(d)) \\
& = & ((l, t(d)), \alpha_d) \\
& = & (b, \beta)_!((l, d, \psi), \id).
\end{eqnarray*}
This completes the proof.
\end{proof}

We will finish this section by showing that {\bf (AMC)} and {\bf (WS)} are reflected by taking sheaves over an internal Grothendieck site $(\ct{C}, {\rm Cov})$, provided every covering sieve is inhabited. Our argument relies on \refprop{constsheaves}; that in turn relies on two lemmas.

\begin{lemm}{compseparatedpresheaf}
Suppose $X$ is a presheaf and $(S, x)$ and $(T, y)$ are two compatible families on $c \in \ct{C}_0$. Then they agree on a common refinement iff for every $(\gamma: d \to c) \in S \cap T$ there is a sieve $R$ covering $d$ such that for every $\delta \in R$ we have $x_{\gamma\delta} = y_{\gamma\delta}$.
\end{lemm}
\begin{proof}
$\Rightarrow$: Suppose $(S, x)$ and $(T, y)$ agree on a common refinement and $\gamma: d \to c$ belongs both to $S$ and $T$. We know that there is a sieve $U$ covering $c$ which refines $S$ and $T$ and on which $x$ and $y$ agree. Pulling back this sieve along $\gamma$ we find a sieve $R$ covering $d$. Then $x$ and $y$ agree on all elements of the from $\gamma\delta$ with $\delta \in R$.

$\Leftarrow$: Put
\[ U = \{ \alpha \in S \cap T \, : \, x_\alpha = y_\alpha \}. \]
We need to show that $U$ is covering. For this purpose, pick $\gamma_0: c_0 \to c \in S$ and $\gamma_1:d \to c_0 \in \gamma_0^* T$. Then $\gamma = \gamma_0\gamma_1 \in S \cap T$ and therefore there is a covering sieve $R$ on $d$ such that for every $\delta \in R$ we have $x_{\gamma\delta} = y_{\gamma\delta}$. In particular, $\gamma^*U$ covers $d$. But then $\gamma_0^* U$ covers $c_0$ and $U$ covers $c$, both times by local character.
\end{proof}

\begin{lemm}{eqclsmall}
Suppose $(\ct{C}, {\rm Cov})$ is an internal Grothendieck site in $\ct{E}$ in which both $\ct{C}_0$ and every covering sieve are inhabited. Let $X$ be an object in \ct{E} and
\[ \pi_!(X \times \ct{C}_0, \pi_2) \]
be the constant presheaf over $X$. Then equivalence classes of compatible families over this constant presheaf which agree on a common refinement are small.
\end{lemm}
\begin{proof}
Note that if $(S, x)$ is one compatible family from $ \pi_!(X \times \ct{C}_0, \pi_2) $ over $c$ and $(T, y)$ is another, then they agree on a common refinement iff we have  $x_{\alpha} = y_{\alpha}$ for every $\alpha \in S \cap T$. This is an immediate consequence of the previous proposition and the fact that every cover is inhabited.

If $(S, x)$ is a compatible family on $c \in \ct{C}_0$, then the collection of compatible families which agree with it on a common refinement is in bijective correspondence with the set
\begin{eqnarray*}
\{ \, T \in {\rm BCov}(c) &  : & (\forall \gamma:d \to c \in T) \, (\forall \delta: e \to d, \delta': e' \to d) \\
& &  \big ( \, \gamma\delta \in S \land \gamma\delta' \in S \to x_{\gamma\delta} = x_{\gamma\delta'} \, \big) \, \}.
\end{eqnarray*}
For if $T$ belongs to a compatible family $(T, y)$ which agrees with $(S, x)$ on a common refinement, and we have maps $\gamma, \delta, \delta'$ in \ct{C} with  $\gamma \in T$, $\gamma\delta \in S$ and $\gamma\delta' \in S$, then
\[ x_{\gamma\delta} = y_{\gamma\delta} = y_{\gamma} = y_{\gamma\delta'} = x_{\gamma\delta'} \]
by the the remark we made at the end of the previous paragraph.

Conversely, if $T$ is a basic covering sieve with this property and $\gamma \in T$, then we can pull back $S$ along $\gamma$; this yields a covering sieve and since covering sieves are inhabited, this means there is a $\delta$ such that $\gamma \delta \in T$. So we may put
\[ y_{\gamma} := x_{\gamma\delta}, \]
which does not depend on the choice of $\delta$ by the assumption on $T$. This yields a compatible family $(T, y)$ which is equivalent to $(S, x)$. Moreover, this construction is clearly inverse to the operation of dropping the $y$ from the $(T, y)$. So we conclude that equivalence classes of compatible families are small, because they are in bijective correspondence with the set above.
\end{proof}

\begin{prop}{constsheaves}
Suppose $(\ct{C}, {\rm Cov})$ is an internal Grothendieck site in $\ct{E}$ in which both $\ct{C}_0$ and every covering sieve are inhabited. Then the functor  $\Delta$ which sends every object $X$ to the sheafification of the constant presheaf over $X$ reflects smallness.
\end{prop}
\begin{proof}
Assume $(\ct{C}, {\rm Cov})$ is an internal Grothendieck site in $\ct{E}$ in which both $\ct{C}_0$ and every covering sieve are inhabited. From the latter assumption it follows that every constant presheaf is separated, so $\Delta(X)$ is obtained by quotienting the compatible families over the constant presheaf over $X$; since the equivalence classes are small by the previous lemma, this implies that the object of compatible families over the constant presheaf on $X$ is small whenever $\Delta(X)$ is. Since the constant families on some object $c \in \ct{C}_0$ (i.e., those $(S, x)$ for which $S$ is the maximal sieve on $c$) can be identified by a bounded formula, $X$ will then be small as well.
\end{proof}

\begin{theo}{Wtypesreflsheaves}
Suppose $(\ct{C}, {\rm Cov})$ is an internal Grothendieck site in $\ct{E}$ in which both $\ct{C}_0$ and every covering sieve are inhabited. If ${\bf (WS)}$ holds in sheaves over $(\ct{C}, {\rm Cov})$, then it also holds in $\ct{E}$.
\end{theo}
\begin{proof}
It is not hard to see that $\Delta$ preserves W-types in the sense that \[ W(\Delta(f)) = \Delta(W(f)). \] Therefore the statement follows from \refprop{constsheaves}.
\end{proof}

\begin{theo}{AMCreflectedsheaves}
Suppose $(\ct{C}, {\rm Cov})$ is an internal Grothendieck site in $\ct{E}$ in which both $\ct{C}_0$ and every covering sieve are inhabited. If ${\bf (AMC)}$ holds in sheaves over $(\ct{C}, {\rm Cov})$, then it also holds in $\ct{E}$.
\end{theo}
\begin{proof}
For once we reason internally. Assume $(\ct{C}, {\rm Cov})$ is a small site in which both $\ct{C}_0$ and every covering sieve are inhabited. Let $X$ be a small object in \ct{E}. Consider $\pi_!(X \times \ct{C}_0, \pi_2)$, the constant presheaf over $X$. This presheaf is separated and hence a dense subobject of its sheafification $
\Delta X := \rho_!(X \times \ct{C}_0, \pi_2)$. For an element $t \in \Delta X(c)$, we will write $t \in X$ if it belongs to this subobject.

We first apply ${\bf (AMC)}$ in the category of sheaves to $\Delta X$; concretely this means that there is diagram
\diag{Y \ar[r] \ar[d] & U \times \Delta X \ar[d] \ar[r] & \Delta X \ar[d]  \\
I \ar@{->>}[r] & U \ar@{->>}[r] & 1 }
in the category of sheaves in which the square on the left is a collection square. By replacing, if necessary, the category of sheaves over $(\ct{C}, {\rm Cov})$ by its slice over $U$ (which is also a category of sheaves), we may assume that $U = 1$. Therefore the diagram above reduces to
\diag{ Y \ar[d] \ar@{->>}[r]^F \ar[d]_G & \Delta X \ar[d] \\
I \ar@{->>}[r] & 1. }
For the moment, fix a pair $c \in \ct{C}_0$ and $i \in I(c)$. Write
\[ \widetilde{Y}_{c, i} = \{ (\alpha: d \to c, y \in Y_i(d))  \, : \,  G_d(y) = i \cdot \alpha, F_d(y) \in X \} \]
and let $F_{c, i}: \widetilde{Y}_{c, i} \to X$ be the map sending $(\alpha: d \to c, y)$ to $F_d(y)$. Since the square above is covering and every covering sieve is inhabited, this map is surjective.

We now use fullness in \ct{E} to find a small collection $\mathcal{A}_{c, i} \subseteq {\rm Pow}(\widetilde{Y}_{c, i})$ such that:
\begin{enumerate}
\item For every element $A \in \mathcal{A}_{c, i}$ the map $F_{c, i} \upharpoonright A: A \to X$ is still surjective.
\item For every small $B \subseteq \widetilde{Y}_{c, i}$, if $F_{c, i} \upharpoonright B: B \to X$ is surjective, then there is an element $A \in \mathcal{A}_{c, i}$ such that $A \subseteq B$.
\end{enumerate}
(Strictly speaking we also need to use the collection axiom to justify writing these small collections $\mathcal{A}_{c, i}$ as a function of $(c, i)$: see \reflemm{semichoicelemma} below.) We claim that
\[ \{ F_{c, i} \upharpoonright A: A \to X \, : \, c \in \ct{C}_0, i \in I(c),  A \in \mathcal{A}_{c, i} \} \]
is a set of surjections onto $X$ as in the statement of ${\bf (AMC)}$.

To see this, let $P \to X$ be an arbitrary surjection. The map $\Delta P \to \Delta X$ is still a surjection, but then in the category of sheaves. Therefore in the category of sheaves there exists a surjection $V \to 1$ and a map $\xi: V \to I$ fitting into a covering square
\diag{ V \times_I Y \ar[r]^(.6)H \ar[d] & \Delta P \ar@{->>}[r] &  \Delta X \ar[d] \\
V \ar@{->>}[rr] & & 1.}
As $V \to 1$ is epi and both $\ct{C}_0$ and every covering sieve are inhabited, we can find elements $c \in \ct{C}_0$ and $v \in V(c)$. Put $i = \xi(v) \in I(c)$ and
\[ B = \{ \, (\alpha: d \to c, y) \in \widetilde{Y}_{c, i} \, : \, H_d(v \cdot \alpha, y) \in P \, \}. \]
The proof will be finished once we show that $F_{c, i} \upharpoonright B: B \to X$ is surjective.

So let $x \in X$. Since $F_{c, i}$ is surjective and every cover is inhabited, this means that there is a pair $(\alpha: d \to c, y) \in \widetilde{Y}_{c, i}$ such that $F_d(y) = x$. Since $H_d(i \cdot \alpha, y) \in \Delta(P)(d)$ and $\pi_!(P \times \ct{C}_0, \pi_2)$ lies dense in $\Delta P$, we find $\beta: e \to d$ with
\[ H_e(i \cdot \alpha\beta, y \cdot \beta) \in P. \]
Therefore $(\alpha\beta, y \cdot \beta) \in B$ and $F_{c, i}(\alpha\beta, y \cdot \beta) = F_e(y \cdot \beta) = x$.
\end{proof}

\section{Relation of AMC to other axioms}

It will be the aim of this section to compare our version of {\bf (AMC)} to other axioms which have appeared in the literature, including the principle called the axiom of multiple choice in \cite{moerdijkpalmgren02} and its reformulation in \cite{rathjen06b}. Throughout this section, our metatheory will be {\bf CZF}.

Before we compare our axiom to the principles in \cite{moerdijkpalmgren02} and \cite{rathjen06b}, we need to make a definition.

\begin{defi}{collfam}
We will say a surjection $p: Y \to X$ \emph{refines} another surjection $q: Z \to X$ if there is a map $f: Y \to Z$ such that $qf = p$. An indexed family $(Y_i)_{i \in I}$ will be called a \emph{collection family} if each surjection $p: E \to Y_i$ is refined by one of the form $q: Y_{i'} \to Y_i$.
\end{defi}

Consider:
\begin{enumerate}
\item The axiom of multiple choice according to \cite{moerdijkpalmgren02}: for every set $X$ there is an inhabited collection family $(Y_i)_{i \in I}$ together with surjections $q_i: Y_i \to X$.
\item A strengthened version of the above: for every set $X$ there exist an inhabited collection family $(Y_i)_{i \in I}$ and surjections $q_i: Y_i \to X$ such that each surjection $p: E \to Y_i$ is refined by a map $q: Y_{i'} \to Y_i$ over $X$.
\item The axiom of multiple choice as reformulated in \cite{rathjen06b}: every set $X$ is a member of a collection family.
\end{enumerate}

\begin{prop}{allstrongversionsequiv}
These principles are all equivalent in {\bf CZF}.
\end{prop}
\begin{proof} (1) $\Rightarrow$ (3): simply add the set $X$ to the collection family.

(3) $\Rightarrow$ (2): if $(Z_k)_{k \in K}$ is a collection family containing $X$, then let $I$ be the collection of all surjections $Z_k \to X$.

(2) $\Rightarrow$ (1) is obvious.
\end{proof}

We will call any of these equivalent principles \emph{strong {\bf (AMC)}}. As the name suggests, it implies our present version of {\bf (AMC)}.

\begin{prop}{MPAMCstronger}
Strong {\bf (AMC)} implies {\bf (AMC)}.
\end{prop}
\begin{proof}
Suppose $X$ is a set and $\{ p_i: Y_i \twoheadrightarrow X \, : \, i \in I \}$ is an inhabited set of surjections as in version 2 of strong {\bf (AMC)}. We claim that $\{ p_i \, : \, i \in I \}$   is also a set of surjections witnessing {\bf (AMC)} in the sense of this paper. To show this, let $f: Z \epi X$ be any surjection. Since $I$ is inhabited, we can pick an element $i \in I$ and construct the pullback:
\diag{T \ar@{->>}[r]^g \ar@{->>}[d] & Y_i \ar@{->>}[d]^{p_i} \\
Z \ar@{->>}[r]_f & X. }
Using the property of $\{ p_i \, : \, i \in I \}$, we find a $j \in J$ and a surjection $h: Y_j \epi Y_i$ factoring through $g$. Therefore $p_j = p_i \circ h$ factors through $f$.
\end{proof}

We expect the converse to be unprovable in {\bf CZF}. However, there is an axiom scheme suggested by Peter Aczel in \cite{aczel08} which implies that our present version of {\bf (AMC)} and strong {\bf (AMC)} are equivalent. This axiom scheme is:
\begin{description}
\item[The Relation Reflection Scheme (RRS):] Suppose $R, X$ are classes and $R \subseteq X \times X$ is a total relation. Then there is for every subset $x \subseteq X$ a subset $y \subseteq X$ with $x \subseteq y$ such that $(\forall a \in y) \, (\exists b \in y) \, (a, b) \in R$.
\end{description}
Our proof of this fact relies on the following lemma:
\begin{lemm}{semichoicelemma}
Suppose $\varphi(x, y)$ is a predicate such that
\[ \varphi(x, y) \land y \subseteq y' \rightarrow \varphi(x, y'). \]
Then, if
\[ (\forall x \in a) \, (\exists y) \, \varphi(x, y), \]
there is a function $f: a \to V$ such that $\varphi(x, f(x))$ for all $x \in a$.
\end{lemm}
\begin{proof}
First use collection to find a set $b$ such that
\[ (\forall x \in a) \, (\exists z \in b) \, \big( \, z = (z_0, z_1) \land z_0 = x \land \varphi(z_0, z_1) \, \big). \]
Then put $f(x) = \bigcup \{ z_1  \, : \, (x, z_1) \in b \}$, which is a set by the union and replacement axioms.
\end{proof}

\begin{prop}{withRRSrealAMC}
Strong {\bf (AMC)} follows from {\bf (AMC)} and {\bf (RRS)}.
\end{prop}
\begin{proof}
Fix a set $X$. We define a relation $R \subseteq {\rm Pow}({\rm Surj}(X)) \times {\rm Pow}({\rm Surj}(X))$ by putting
\begin{quote}
$(\alpha, \beta) \in R$ iff for every $f: Y \to X \in \alpha$ and every surjection $g: Z \to Y$ there are $h: T \to X \in \beta$, $p: T \epi Y$ and $k: T \to Z$ fitting into a commutative diagram as follows:
\diag{ T \ar[d]_k \ar[r]^{h} \ar@{>>}[rd]^p & X \\
Z \ar@{->>}[r]_g & Y. \ar[u]_{f} }
\end{quote}
It follows from {\bf (AMC)} that $R$ is total: for if $\alpha$ is any set of surjections onto $X$, then {\bf (AMC)} implies that for every $f: Y \to X \in \alpha$ there is a set of surjections onto $Y$ such that any such is refined by one in this set. By applying the previous lemma to this statement, we find for every $f \in \alpha$ a set $A_f$ of surjections with this property. We find our desired $\beta$ as $\beta = \{f \circ g \, : \, g \in A_f \}$.

By applying {\bf (RRS)} to $R$, we obtain a set $M \subseteq {\rm Pow}({\rm Surj}(X))$ such that $\{ \id_X: X \to X \} \in M$ and $(\forall \alpha \in M) \, (\exists \beta \in M) \, (\alpha, \beta) \in R$. Put $N = \bigcup M$. It is straightforward to check that $N$ is a set of surjections witnessing strong {\bf (AMC)}.
\end{proof}

Note that the following was shown in \cite{moerdijkpalmgren02}:

\begin{theo}{realAMCwithWSimpliesREA} {\rm \cite[Theorem 7.1(ii)]{moerdijkpalmgren02}}
The regular extension axiom {\bf (REA)} follows from the combination of strong {\bf (AMC)} and {\bf (WS)}.
\end{theo}

\noindent
We expect this theorem to fail if one replaces strong {\bf (AMC)} with our present version of {\bf (AMC)}. (In fact, this is the only application of strong {\bf (AMC)} we are aware of that probably cannot be proved using our weaker version.) We do not consider this a serious drawback of our present version of {\bf (AMC)} or our proposal to extend {\bf CZF} with {\bf (WS)} and this axiom, because the main (and, so far, only) application of {\bf (REA)} is the Set Compactness Theorem, which, as we showed in Section 2, \emph{is} provable using {\bf (WS)} and the present version of {\bf (AMC)}.

\bibliographystyle{plain} \bibliography{ast}

\end{document}